\documentclass[12pt]{amsart}
\usepackage{latexsym,amsxtra,amscd,ifthen,amsmath,color, multicol,hyperref,mathdots,mathrsfs}
\usepackage{amsfonts}
\usepackage{verbatim}
\usepackage{amsmath}
\usepackage{amsthm}
\usepackage{amssymb}
\usepackage[all,cmtip]{xy}
\usepackage[normalem]{ulem}
\usepackage{enumerate}
\usepackage[latin1]{inputenc}
\usepackage{amsmath}
\usepackage{amsfonts}
\usepackage{amssymb}
\usepackage{graphicx}
\usepackage{multirow}
\usepackage{latexsym}
\swapnumbers
\theoremstyle{plain}
\newtheorem{thm}{Theorem}[section]
\newtheorem{lem}[thm]{Lemma}
\newtheorem{prop}[thm]{Proposition}
\newtheorem{cor}[thm]{Corollary}

\theoremstyle{definition}

\newtheorem*{Ack}{Acknowledgement}
\newtheorem{dfn}[thm]{Definition}
\newtheorem{rem}[thm]{Remark}

\theoremstyle{remark}

\def\k{\ensuremath{k}}

\newcommand*{\e}{\ensuremath{\varepsilon}}

\newcommand*{\field}{\ensuremath{\bold{k}}}

\newcommand*{\gr}{\ensuremath{\text{\upshape gr}}}

\def\dim{\operatorname{dim}}

\def\Chara{\operatorname{char}}

\begin{document}
\thispagestyle{empty}

\title[Indicators in positive characteristic]{Indicators of Hopf algebras in positive characteristic}

\author{Linhong Wang}
\author{Xingting Wang}

\address{Department of Mathematics\\
University of Pittsburgh, Pittsburgh, PA 15260}
\email{lhwang@pitt.edu}

\address{Department of Mathematics\\
Temple University,
Philadelphia, PA 19122}
\email{xingting@temple.edu}

\thanks{Research of the second author is supported by the AMS-Simons Travel Grant}



\begin{abstract}
The notion of $n$-th indicator for a finite-dimensional Hopf algebra was introduced by Kashina, Montgomery and Ng as gauge invariance of the monoidal category of its representations. The properties of these indicators were further investigated by Shimizu. In this short note, we show that the indicators appearing in positive characteristic all share the same sequence pattern if we assume the coradical of the Hopf algebra is a local Hopf subalgebra.
\end{abstract}

\maketitle

\section{Introduction}

Let $G$ be a finite group. Classically, the criterion to determine the type of an irreducible representation of $G$ as real, complex or quaternionic is to find the value of Frobenius-Schur indicator of the representation as $1$, $0$, or $-1$, respectively. Let $H$ be a semisimple Hopf algebra over an algebraically closed field of characteristic zero, and let $V$ be an irreducible representation of $H$. Linchenko and Montgomery in \cite{LinMon} generalized this classical group theoretic result to $H$, that is, the possible values $1$, $0$, or $-1$ of the second Frobenius-Schur indicator $\nu_2(V)$ provide certain information of the irreducible representation $V$ of $H$.  In \cite{KSZ2}, Kashina-Sommerh\"auser-Zhu further provided a sequence of higher Frobenius-Schur indicators $\nu_n(V)$ for $n \geq 2$. 

In \cite{KMN}, Kashina-Montgomery-Ng proposed a definition of higher indicators for a nonsemisimple finite-dimensional Hopf algebra $H$ using its regular representation and proved that these indicators are gauge invariant with respect to all gauge transformations, that is, $\nu_n(H)=\nu_n(H^F)$ up to a Drinfeld twist given by a $2$-cocycle gauge transformation $F$.

The definition of indicators of the regular representation of a finite-dimensional Hopf algebra is straightforward by taking the trace of the Sweedler powers followed by the antipode; see \cite[Definition 2.1]{KMN}. We notice that the definition of indicators is well-behaved in positive characteristic and that many of their properties remain valid including gauge invariance. For example, if $H=\k\, G$ is the group algebra of a finite group $G$ over a base field $\k$, then the $n$-th indicator is given by
\begin{align}\label{IndG}
\nu_n(\k\, G)=\#\{g\in G\, |\, g^n=1_G\}\, \mod  \Chara\, \k.
\end{align}
Now suppose $G$ is a finite $p$-group. It is easy to see that $\nu_n(\k\, G)=1$ whenever $\text{gcd}(p,n)=1$. The difficulty lies in the case when $p\mid n$, where the actual value of $\nu_n(\k\, G)$ depends on the structure of the $p$-group $G$. But we can always show that $p\mid \nu_n(\k\, G)$ if $p \mid n$. So if we let $\Chara\, \k=p$, then $\nu_n(\k\, G)=1$ whenever $\text{gcd}(p,n)=1$ and $\nu_n(\k\, G)=0$ elsewhere. Another example of Hopf algebra over characteristic $p$ with the same sequence of indicators is the Radford algebra $R(p)$; see \cite[Theorem 3.3]{HH}. We refer to such a sequence as \emph{the $p$-pertinent sequence}; see Definition \ref{D: ppertinent}. These examples prompt us to study indicators in the modulo $p$ setting, or simply over a base field $\k$ of prime characteristic $p$.

In this short note, we study any finite-dimensional Hopf algebra having the local dual Chevalley property, which requires the coradical of the Hopf algebra to be a local Hopf subalgebra; see Definition \ref{DChev}. The group algebra $\k\, G$ and the Radford algebra $R(p)$ discussed above are examples of such Hopf algebras. Our main result, Theorem \ref{Main}, shows that the sequence of indicators of a finite-dimensional Hopf algebra having the local dual Chevalley property is the $p$-pertinent sequence. The approach of our proof is to reduce the problem firstly to the case of restricted universal enveloping algebras by applying the associated graded algebras in Lemma \ref{GrC} and then to the case of $p$-groups in Lemma \ref{RU}. We also show, in Lemma \ref{L:ppertinent}, that the minimal polynomial of the $p$-pertinent sequence is $f(x)=x^p-1$. These results particularly apply to any pointed Hopf algebras of dimension $p^n$ for some integer $n$; see Corollary \ref{C2}. Lastly, as noted in \cite{Shi}, the sequence of indicators can be extended uniquely to have negative terms. Our results then hold for the sequence of indicators with all integer indices.

\begin{Ack} This work started during the second author's visit in the math department at the University of Pittsburgh in March 2016. We would like to express our gratitude to the department.
\end{Ack}

\section{Preliminary}\label{S:1}
In this section, we will review some basic definitions and properties of Hopf algebras and their indicators.
Let $\k$ be a base field of arbitrary characteristic, and $H$ be a finite-dimensional Hopf algebra over $\k$.
The unadorned tensor $\otimes$ means $\otimes_\k$. We use the standard notation $(H,\,m,\,u,\,\Delta,\,\e,\,S)$, where $m: H\otimes H \to H$ is the multiplication map, $u: \k\to H$ is the unit map, $\Delta: H\to H\otimes H$ is the comultiplication map, $\e: H\to \k$ is the counit map, and $S: H\to H$ is the antipode. The vector space dual of $H$ is again a Hopf algebra and will be denoted by $H^*$. The bialgebra maps and the antipode of $H^*$ are given by $(m_{H^*},\,u_{H^*},\,\Delta_{H^*},\,\e_{H^*},\,S_{H^*})=(\Delta^*,\,\e^*,\,m^*,\,u^*,\,S^*)$, where $^*$ is the transpose of linear maps. We use the Sweedler's notation (sumless) $\Delta(h)=h_{(1)}\otimes h_{(2)}$ for any $h\in H$. In the dual Hopf algebra $H^*$, given $f,\,g\in H^*$, then $fg(h)=f\left(h_{(1)}\right)g\left(h_{(2)}\right)$ for any $h\in H$ and $\e_{H^*}(f)=f(1)$.





\begin{dfn}\cite[Definition 5.1.5, 5.2.1]{Mon}
\
\begin{itemize}
\item[(i)] The \emph{coradical} $H_0$ of $H$ is the sum of all simple subcoalgebras of $H$.
\item[(ii)] We say $H$ is \emph{connected} if $H_0$ is one-dimensional.
\item[(iii)] The \emph{coradical filtration} $\{H_n\}_{n\ge 0}$ of $H$ is defined inductively such that $H_n=\Delta^{-1}(H\otimes H_{n-1}+H_0\otimes H)$ for all integers $n\ge 1$. We use $\gr_CH=\bigoplus_{i\ge 0} H_{i}/H_{i-1}$ ($H_{-1}=0$) to denote the associated graded coalgebra with respect to the coradical filtration of $H$.
\item[(iv)] We denote by $\gr_JH:=\bigoplus_{i\ge 0} J^i/J^{i+1}$ ($J^0=H$) the associated graded algebra with respect to the $J$-adic filtration of $H$, where $J$ is the Jacobson radical of $H$.
\end{itemize}
%
\end{dfn}

\begin{prop}\label{Gr}
Let $H$ be a finite-dimensional Hopf algebra. The following are equivalent.
\begin{itemize}
\item[(i)] $\gr_CH$ is a graded Hopf algebra.
\item[(ii)] $\gr_J (H^*)$ is a graded Hopf algebra.
\item[(iii)] $H_0$ is a Hopf subalgebra.
\item[(iv)] The Jacobson radical $J$ of $H^*$ is a Hopf ideal.
\end{itemize}
In these cases, we have $\gr_CH\cong (\gr_J(H^*))^*$ as graded Hopf algebras.
\end{prop}
\begin{proof}
(i)$\Leftrightarrow$ (iii) It follows from \cite[Lemma 5.2.8]{Mon}.

(i) $\Leftrightarrow$ (ii) We apply the natural isomorphism $H_i\cong (H^*/J^{i+1})^*$ \cite[Lemma 2.2 (ii)]{Wang2014} to obtain
\[
H_{i}/H_{i-1}\cong (H^*/J^{i+1})^*/(H^*/J^i)^*\cong (J^i/J^{i+1})^*
\]
for all $i\ge 0$. Thus
\[
(\gr_J(H^*))^*=\bigoplus_{i\ge 0}\, (J^{i}/J^{i+1})^*\cong \bigoplus_{i\ge 0}\, H_i/H_{i-1}=\gr_CH.
\]
So $\gr_CH$ is a graded Hopf algebra if and only if $(\gr_CH)^*$ is a graded Hopf algebra if and only if $\gr_J(H^*)$ is a graded Hopf algebra.

(iii)$\Leftrightarrow$(iv) It is similar since $H_0\cong H^*/J$ \cite[Lemma 2.2 (ii)]{Wang2014}.
\end{proof}

\begin{dfn}\cite[Definition 3.1]{Shi}\label{D:Sweedler}
Let $H$ be a Hopf algebra. For an integer $m$, the \emph{$m$-th Sweedler power map}
\[
P^{(m)}_H: H\to H,\quad h\mapsto h^{[m]}\quad (h\in H)
\]
is defined to be the $m$-th power of the identity map $\mathrm{id}_H: H\to H$ with respect to the convolution product. Thus, for all $h\in H$ and an integer $m\ge 1$,
\[
h^{[0]}=\epsilon(h)1_H,\quad h^{[m]}=h_{(1)}\cdots h_{(m)}\quad \text{and}\quad h^{[-m]}=S(h_{(1)})\cdots S(h_{(m)}).
\]
For any integer $n\in \mathbb Z$, we define the \emph{$n$-th indicator} of $H$ by
\[
\nu_n(H)=\mathrm{Tr}\, (S_H\circ P_H^{n-1}: H\to H),
\]
where $\mathrm{Tr}$ means the ordinary trace of a linear operator.
\end{dfn}

\begin{rem}
\
\begin{itemize}
\item[(i)] The indicator $\nu_n(H)$ is the same as in \cite[Definition 2.1]{KMN} for $n\ge 1$. One checks that $\nu_0(H)=\mathrm{Tr}(S^2_H)$, $\nu_1(H)=1$ and $\nu_2(H)=\mathrm{Tr}(S_H)$.

\item[(ii)] Let $H$ and $K$ be two finite-dimensional Hopf algebras over $\field$ such that the two representation categories $\text{Rep}(H)$ and $\text{Rep}(K)$ are monoidally equivalent. By \cite[Theorem 2.2]{NgSch},  $H\cong K^F$, where $K^F$ is a Drinfeld twist by a gauge transformation $F$ on $H$ which satisfies some $2$-cocycle conditions. Then $H$ and $K$ are said to be \emph{gauge equivalent} Hopf algebras. According to \cite[Theorem 2.2]{KMN} and \cite[Theorem 3.10]{Shi}, $\nu_n$ is a gauge invariant.
\end{itemize}
\end{rem}

Fix an algebraic closure $\overline{k}$ of $k$. For each integer $N\ge 1$, we define $\mathcal O_N$ to be the subring of $\overline{k}$ generated by the roots of the polynomial $X^N-1$. By definition, each element of $\mathcal O_N$ is of the form $z=a_1\omega_1+\cdots+a_m\omega_m$ for some $a_i\in \mathbb Z$ and some $\omega_i\in \overline{k}$ such that $\omega_i^N=1$. For such an element $z$, we set
\[
\overline{z}:=a_1\omega_1^{-1}+\cdots+a_m\omega_m^{-1}\in \mathcal O_N\quad \text{and}\quad |z|^2:=z\cdot \overline{z}.
\]
The assignment $z\mapsto \overline{z}$ is a well-defined ring automorphism of $\mathcal O_N$.

\begin{prop}\cite[Corollary 3.17]{Shi}\label{PIN}
Let $H$ and $H'$ be finite-dimensional Hopf algebras. Then, for all $n\in \mathbb Z$, we have the following:
\begin{itemize}
\item[(i)] $\nu_n(H^*)=\nu_n(H)$ and $\nu_n(H\otimes H')=\nu_n(H)\cdot \nu_n(H')$.
\item[(ii)] $\nu_n(H)\in k\cap \mathcal O_N$ for all $n\ge 1$, where $N=n\cdot \mathrm{ord}(S_H^2)$.
\item[(iii)] $\nu_n(H^{\mathrm{op}})=\nu_n(H^{\mathrm{cop}})=\overline{\nu_n(H)}$.
\item[(iv)] If $H$ is filtered, then $\nu_n(\gr H)=\nu_n(H)$.
\end{itemize}
\end{prop}

Most of time, a sequence $\{a_n\}$ refers to a nonnegative indexed sequence $\{a_n\}_{n\geq 0}$. In the nature of this paper, we assume that a sequence $\{a_n\}$ also include the negatively indexed terms.

\begin{dfn}
A sequence $\{a_n\}$ is \emph{linearly recursive} if there exists a non-zero polynomial $f(x)=f_0+f_1x+f_{m-1}x^{m-1}+f_mx^m\in k[x]$ such that
\[
f_0a_n+f_1a_{n+1}+\cdots+f_ma_{m+n}=0,
\]
for any integer $n$. In such a case, we say that $\{a_n\}$ satisfies the polynomial $f(x)$. The monic polynomial of the least degree satisfied by a linearly recursive sequence is called \emph{the minimal polynomial} of that sequence.
\end{dfn}

Note that any periodic sequence $\{a_n\}$ is linearly recursive and  is uniquely determined by its positively indexed part, i.e., $\{a_n\}_{n\geq 1}$.

\begin{dfn}\label{D: ppertinent}
Let $p$ be a prime. We refer to a periodic sequence $\{a_n\}$ as \emph{the $p$-pertinent sequence} if
\[\{a_n\}_{n\geq 1}=\Big\{\underbrace{1, \ldots, 1}_{p-1}, 0, \underbrace{1, \ldots, 1}_{p-1}, 0, \ldots\Big\}.\]
\end{dfn}

\begin{lem}\label{L:ppertinent}
Let $p$ be a prime and $\Chara k=q$. i) The minimal polynomial of the $p$-pertinent sequence is
\[f(x)=
\begin{cases}
x^p-1 & \text{if}\ q\nmid p-1\\
x^{p-1}+x^{p-2}+\cdots +x+1 & \text{if}\ q\mid p-1.
\end{cases}
\]
ii) If $b_0=0$, $b_j=(-1)^{j+1}$ and $b_p=b_{p+1}=\ldots=0$, then the sequence
\[\Big\{B_n:=\sum_{j=0}^n {n \choose j} b_j\Big\}_{n\geq 1}\]
is the $p$-pertinent sequence.
\end{lem}

\begin{proof}
i) Note that the sum of any $p$ consecutive terms of the sequence is always $p-1$. Then, when $q\mid p-1$, we see that the sequence satisfies the polynomial $x^{p-1}+x^{p-2}+\cdots +x+1$, and that $p>2$. Suppose that the sequence satisfies a polynomial of lower degree, say $g(x)=g_{p-2}x^{p-2}+ \cdots + g_1x_1+ g_0$ where $g_i \in k$. Then
\begin{align}\label{MPLRS}
A[g_0\, g_1\, \ldots\, g_{p-2}]^{\mathrm{T}}=0,
\end{align}
for the matrix $A=[a_{i,j}]_{p\times (p-1)}$, where
\[a_{i,j}=
\begin{cases}
0& \text{if}\ i=p-j+1\ \text{and}\ 1\leq j\leq p-1,\\
1 & \text{elsewhere}.
\end{cases}
\]
Eliminating the first row of the matrix $A$, we obtain a matrix $A'$ with 0's on the anti-diagonal and 1's elsewhere. The determinant of $A'$ is $-(p-2)$, not divisible by $q$. Therefore the only solution to the system corresponding to $A'$ is (0, \ldots,0). The system  \eqref{MPLRS} has one more equation $\sum_{i=0}^{p-2}g_i=0$ and so only has the solution (0, \ldots,0). Thus, when $q\mid p-1$,  no polynomial of degree less than $p-1$ can be satisfied by the sequence. Therefore, $x^{p-1}+x^{p-2}+\cdots +x+1$ is the minimal polynomial of the $p$-pertinent sequence when $q\mid p-1$. The case for $q\nmid p-1$ can be shown similarly.

ii) Note that \[B_n:=\sum_{j=0}^n {n \choose j} b_j
=\begin{cases}
\sum_{j=1}^n(-1)^{j+1}{n \choose j} & 1\leq n\leq p-1\\
\sum_{j=1}^{p-1}(-1)^{j+1}{n \choose j} & n\geq p.
\end{cases}
\]
Consider
\[(1-x)^n=\sum_{j=0}^n{n \choose j}(-1)^j x^j=:\alpha_0+\alpha_1 x+\ldots+ \alpha_n x^n.\]
Then we have
\[B_n=
\begin{cases}
-(\alpha_1+\ldots+\alpha_n)&  \text{for } 1\leq n \leq p-1\\
-(\alpha_1+\ldots+\alpha_{p-1}) &  \text{for } n \geq p.
\end{cases}
\]
Clearly $B_n=\alpha_0=1$ for $ 1\leq n \leq p-1$ and $B_p=\alpha_0+\alpha_p=0$. That is,
\[\{B_n\}_{ 1\leq n\leq p}=\{1, \ldots, 1, 0 \}.\]
Now consider \[(1-x)^{n+p}=\sum_{j=0}^{n+p}{n+p \choose j}(-1)^j x^j=:\beta_0+\beta_1x+\ldots+\beta_n x^n+\ldots+\beta_{n+p}x^{n+p}.\]
Then
\[B_{n+p}=\sum_{j=1}^{p-1}(-1)^{j+1}{n+p \choose j}=-(\beta_1+\beta_2+\ldots+\beta_{p-1}).\]
Note that
\[(1-x)^{n+p}=(1-x)^n(1-x)^{p}=(1-x)^n-x^p(1-x)^n.\]
This implies that $\beta_{n+1}=\ldots=\beta_{p-1}=0$ if $n\leq p-1$ and that $\alpha_i=\beta_i$ for $1\leq i\leq n\leq p-1$ or $1\leq i \leq p-1 <n$.
Thus, $B_{n+p}=B_n$ for any $n\geq 1$. Therefore, the sequence $\{B_n=\sum_{j=0}^n {n \choose j} b_j\}_{n\geq 1}$ is just the $p$-pertinent sequence.
\end{proof}

The following result is a restatement of \cite[Proposition 2.7]{KMN} and \cite[Corollary 4.6]{Shi}. We include the proof for the sake of completeness.

\begin{prop}
The sequence $\{\nu_n(H)\}_{n\in \mathbb Z}$ is linearly recursive and the degree of its minimal polynomial is at most $(\dim H)^2$. In particular when $\mathrm{char} (k)>0$, the sequence $\{\nu_n(H)\}_{n\in \mathbb Z}$ is periodic.
\end{prop}
\begin{proof}
We recall that $A=\mathrm{End}_k(H)$ is an associative algebra with respect to the convolution product, where the identity map $\mathrm{Id}$ and antipode $S$ are inverse to each other. By Definition \ref{D:Sweedler}, for any integer $n\in \mathbb Z$, the $n$-th power Sweedler map $P_H^{(n)}: H\to H$ can be expressed by the $n$-th power of $\mathrm{Id}$ in $A$. Since $A$ is finite-dimensional, we see that the sequence
\begin{align*}
P_H:=\{P_H^{(n)}\}_{n\in \mathbb Z}
\end{align*}
satisfies a minimal polynomial $\Phi_H(X)=a_mX^m+a_{m-1}X^{m-1}+\cdots a_0\in k[X]$ such that
\begin{align}\label{LRS}
a_mP_H^{(n+m)}+a_{m-1}P_H^{(n+m-1)}+\cdots a_0P_H^{(n)}=0\quad \text{in $A$ for all integers $n$}.
\end{align}
Next we take $\mathrm{Tr}(S\circ \eqref{LRS})$ to obtain
\begin{align}\label{LRSI}
a_m\nu_{n+m}(H)+a_{m-1}\nu_{n+m-1}(H)+\cdots a_0\nu_{n}(H)=0\quad \text{for all integers $n$}.
\end{align}
Moreover, one verifies that  $\mathrm{deg}\, \Phi_H\le \dim A=(\dim H)^2$. This shows the first part. Next we will only prove for the periodicity of the indicators in positive characteristic, and the linear recursive property follows from \cite[Proposition 2.7]{KMN} by the same argument. By \cite[Corollary 4.6]{Shi}, we know $\{\nu_n(H)\}_{n\ge 1}$ is periodic. It is important to point out that $a_0\neq 0$ in \eqref{LRSI} (otherwise the sequence satisfies $\Phi_H(X)/X$ which has smaller degree since $P_H^{(n)}$ is invertible in $A$). Hence one has
\begin{align*}
\nu_n(H)=-a_0^{-1}(a_m\nu_{n+m}(H)+a_{m-1}\nu_{n+m-1}(H)+\cdots a_1\nu_{n+1}(H))\quad \text{for all integers $n$}.
\end{align*}
An easy inductive argument yields that $\{\nu_n(H)\}_{n\in \mathbb Z}$ is periodic.
\end{proof}

\section{Indicators in positive characteristic}
In the remaining of this paper, $\k$ is a base field of prime characteristic $p$. We will study indicators of finite-dimensional Hopf algebras $H$ over $\k$.

\begin{dfn}\label{DChev}
\
\begin{itemize}
\item[(i)] We say that a Hopf algebra $H$ has the \emph{Chevalley property} if the Jacobson radical $J$ of $H$ is a Hopf ideal. Moreover, we say that $H$ has the \emph{connected Chevalley property} if additionally $H/J$ is a connected coalgebra.
\item[(ii)] By duality, we say that a Hopf algebra $H$ has the \emph{dual Chevalley property} if its coradical $H_0$ is a Hopf subalgebra. Additionally if $H_0$ is a local algebra then $H$ is said to have the \emph{local dual Chevalley property}.
\end{itemize}
\end{dfn}

\begin{prop}\label{Chev}
Let $H$ and $H^*$ be finite-dimensional Hopf algebras over $\k$ dual to each other. Then the following are equivalent.
\begin{itemize}
\item[(i)] $H$ has the local dual Chevalley property.
\item[(ii)] $H^*$ has the connected Chevalley property.
\item[(iii)] $H$ has the dual Chevalley property and the graded Hopf algebra $\gr_CH$ is local.
\item[(iv)] $H^*$ has the Chevalley property and the graded Hopf algebra $\gr_J(H^*)$ is connected.
\item[(v)] $H\otimes \overline{\k}$ is pointed and the group of its grouplike elements is a $p$-group.
\item[(vi)] $H^*\otimes  \overline{\k}$ is basic with $p^m$ nonisomorphic simple modules for some integer $m$.
\end{itemize}
In these cases, we have $\dim H=\dim H^*=p^n$ for some integer $n\ge 0$.
\end{prop}
\begin{proof}
(i) $\Leftrightarrow$ (ii) We use the fact that $H_0$ is local if and only if $H^*/J$ is connected \cite[Lemma 2.2 (iii)]{Wang2014}. Thus it follows from Proposition \ref{Gr} (iii) $\Leftrightarrow$ (iv).

(iii) $\Leftrightarrow$ (iv) \& (v) $\Leftrightarrow$ (vi) They can be proved by the same argument.

(i) $\Rightarrow$ (iii) It suffices to show that $\gr_CH$ is local.  Let $I$ be the augmented ideal of $\gr_CH$. Then $I/(\gr_CH)_{\ge 1}\cong H_0^+$. Since $H_0$ is local we have $(H_0^+)^n=0$ for some integer $n$ by \cite[Lemma 2.2 (iii)]{Wang2014}. Moreover, the fact that $\gr_CH$ is a finite-dimensional graded algebra implies that $((\gr_CH)_{\ge 1})^m=0$ for some integer $m$. Hence $I^{nm}=((\gr_CH)_\ge 1)^m=0$. So $\gr_CH$ is local by \cite[Lemma 2.2 (iii)]{Wang2014} again. (iii) $\Rightarrow$ (i) It is clear since the Hopf subalgebra $H_0\subseteq \gr_CH$ is local by \cite[Lemma 2.8 (i)]{Wang2014}.

(i) $\Rightarrow$ (v) By \cite[Lemma 2.8 (ii), (iii)]{Wang2014}, one sees that $(H\otimes \overline{\k})_0=H_0\otimes \overline{\k}$ is a cosemisimple local Hopf algebra. Thus $(H\otimes \overline{\k})_0^*$ is a semisimple connected Hopf algebra over $\overline{\k}$. Hence \cite[Theorem 0.1]{Ma} implies that $(H\otimes \overline{\k})_0\cong \overline{\k}\, G$ for some finite $p$-group $G$.  (v) $\Rightarrow$ (i) It is clear that $(H\otimes \overline{\k})_0\cong \overline{\k}\, G$ is local since $G$ is a finite $p$-group. By \cite[Lemma 2.8 (i), (iii)]{Wang2014}, we know $H_0=H_0\otimes 1\subseteq (H\otimes \overline{\k})_0$ is a local Hopf subalgebra of $H$.

Finally, $\dim H=\dim H^*=\dim \gr_J H^*=p^n$ for some integer $n\ge 0$ by \cite[Proposition 2.2 (7)]{Wang}.
\end{proof}

\begin{prop}\label{Class}
The class of all finite-dimensional Hopf algebras over $\k$ with $p$-pertinent sequence of indicators is closed under taking vector space dual, tensor product, opposite algebra, opposite coalgebra, and associated graded algebra.
\end{prop}
\begin{proof}
It directly follows from Proposition \ref{PIN}.
\end{proof}

\begin{lem}\label{RU}
Let $u(\mathfrak g)$ be the restricted universal enveloping algebra of some finite-dimensional restricted Lie algebra $\mathfrak g$ over $\k$.  Then the sequence $\big\{\nu_n(u(\mathfrak g))\big\}_{n\in \mathbb Z}$ is $p$-pertinent.
\end{lem}
\begin{proof}
We consider the associated graded algebra $\gr_C (u(\mathfrak g))$ with respect to the coradical filtration of $u(\mathfrak g)$. One can verify easily using the restricted PBW basis of $u(\mathfrak g)$ that
\begin{align}\label{local}
\gr_C(u(\mathfrak g))\cong u(\mathfrak h),
\end{align}
where $\mathfrak h$ is a $p$-nilpotent abelian restricted Lie algebra such that $\dim \mathfrak h=\dim \mathfrak g=:d$. For the sake of convenience, we define
\[H(\delta):=\k[x]/(x^p-\delta x),\quad \text{with}\quad \delta=0,1,\]
where the coalgebra structure is given by $\Delta(x)=x\otimes 1+1\otimes x$. Thus \eqref{local} implies that $\gr_C (u(\mathfrak g))\cong H(0)^{\otimes d}$. Hence it reduces to prove that $\{\nu_n(H(0))\}_{n\ge 1}$ is $p$-pertinent by Proposition \ref{Class}. One further observes that
\[
H(0)\cong \gr_CH(1)\cong k^{C_p}.\]
Then it remains to show that $\{\nu_n(\k^{C_p})\}_{n\ge 1}$ or $\{\nu_n(\k\, C_p)\}_{n\ge 1}$ is $p$-pertinent. This has been already discussed by \eqref{IndG} as
\[\nu_n(\k\, C_p)=\#\{g\in C_p\, |\, g^n=1_G\}\, \mod p.\]
\end{proof}

\begin{lem}\label{GrC}
Let $H$ be a finite-dimensional connected Hopf algebra over $\k$. Then $(\gr_CH)^*\cong u(\mathfrak g)$ for some finite-dimensional restricted Lie algebra $\mathfrak g$ over $\k$.
\end{lem}
\begin{proof}
It is well-known that $\gr_CH$ is a connected coradically graded Hopf algebra. By \cite[Theorem 3.1]{Wang}, the algebra structure of $\gr_CH$ is given as follows
\[
\gr_CH\cong \k[x_1,x_2,\dots,x_d]/(x_1^p,x_2^p,\dots,x_d^p),
\]
where $\dim H=p^d$ for some integer $d$. Then one sees that $\gr _CH$ is local and $\dim J/J^2=d$ where $J=(\gr_C H)_{\ge 1}$ is the Jacobson radical of $\gr_CH$. By duality again, one can conclude that $(\gr_CH)^*$ is connected, whose primitive space has dimension $d$. So $(\gr_CH)^*$ must be primitively generated and is isomorphic to $u(\mathfrak g)$ for some finite-dimensional restricted Lie algebra $\mathfrak g$.
\end{proof}

\begin{thm}\label{Main}
Let $H$ be any finite-dimensional Hopf algebra over $\k$ having the local dual Chevalley property or connected Chevalley property. Then the sequence $\{\nu_n(H)\}_{n\in \mathbb Z}$ is $p$-pertinent.
\end{thm}
\begin{proof}
In view of Proposition \ref{Chev}, it suffices to assume that $H$ has the connected Chevalley property. Moreover, we can apply Lemma \ref{GrC} to obtain
\[
\nu_n(H)=\nu_n(\gr_JH)=\nu_n\left(\gr_C(\gr_J H)\right)=\nu_n\left((\gr_C(\gr_J H))^*\right)=\nu_n(u(\mathfrak g))
\]
for some finite-dimensional restricted Lie algebra $\mathfrak g$ over $\k$. Hence the $p$-pertinency follows from Lemma \ref{RU}. The minimal polynomial is derived from Lemma \ref{L:ppertinent}.
\end{proof}

\begin{rem}
Let $H$ be a finite-dimensional Hopf algebra with $\{\nu_n(H)\}_{n\geq 1}$ being the $p$-pertinent sequence. Let $b_j$'s as stated in Lemma \ref{L:ppertinent}. Then the constant sequences $c_j=\{c_j(n)=b_j\}_{n\geq 1}$ for $j=0,1, 2, \ldots$ are the unique series of sequences discussed in \cite[Theorem 4.4]{Shi} such that \[\nu_n(H)=\sum_{j=0}^{\infty}{n \choose j}c_j(n).\]
\end{rem}

\begin{cor}\label{C1}
Let $H$ or $H^*$ be a pointed finite-dimensional Hopf algebra over $\k$, where $G(H)$ is a $p$-group. Then the sequence $\big\{\nu_n(H)\big\}_{n\in \mathbb Z}$ is $p$-pertinent.
\end{cor}
\begin{proof}
It follows from Proposition \ref{Chev} and Theorem \ref{Main}.
\end{proof}

\begin{cor}\label{C2}
Let $H$ or $H^*$ be a pointed finite-dimensional Hopf algebra over $\k$ of dimension $p^n$ for some integer $n$. Then the sequence $\big\{\nu_n(H)\big\}_{n\in \mathbb Z}$ is $p$-pertinent.
\end{cor}
\begin{proof}
We can apply the previous corollary noting that $|G(H)|$ or $|G(H^*)|$ divides $\dim H=p^n$ by Nichols-Zoeller freeness theorem \cite{NZ}.
\end{proof}

\begin{rem}
\
\begin{itemize}
\item[(i)] The classification of pointed Hopf algebras of dimension $p^n$ for $n\le 3$ has been completed in \cite{NW, NWW, Wang, WW} over an algebraically closed field of prime characteristic $p$. By the result above, we know the indicators of these Hopf algebras are all $p$-pertinent.
\item[(ii)] For any finite-dimensional Hopf algebra $H$ that has the local dual Chevalley property or connected Chevalley property, one can use the same argument in this paper to show that
\[
\mathrm{Tr}(S^n)=0\quad \text{if}\ p=2\quad \text{or}\quad \mathrm{Tr}(S^n)=\begin{cases}
0  &  n\equiv 0  \pmod 2\\
1  &  n\equiv 1 \pmod 2
\end{cases}
\quad \text{if}\ p>2.
\]
\end{itemize}
\end{rem}



\begin{thebibliography}{99}
\bibitem{HH} H.~Hu,  X.-Y.~Hu, L.-H.~Wang, and X.-T.~Wang, Computing indicators of Radford algebras, accepted, to appear \emph{Involve, a journal of mathematics}.

\bibitem{KMN} Y.~Kashina, S.~Montgomery, and S.-H.~Ng, On the trace of the antipode and higher indicators, \emph{Israel~J.~Math.}, 188 (2012), 57--89.


\bibitem{KSZ2} Y.~Kashina, Y.~Sommerh\"auser and Y.~Zhu, On higher Frobenius-Schur indicators, Memories of the American Mathematical Society, 181 (2006), no. 855, viii+65 pp.
    
\bibitem{LinMon} V.~Linchenko and S.~ Montgomery, A Frobenius-Schur theorem for Hopf algebras, \emph{Algebr.~Represent.~Theory}, 3 (2000), 347--355. 

\bibitem{NgSch} S.-H. Ng and P. Schauenburg, Central invariants and higher indicators for semisimple quasi-Hopf algebras, \emph{Trans. Amer. Math. Soc.}, 360 (2008), 1839--1860. 

\bibitem{Ma} A.~Masuoka, Semisimplicity criteria for irreducible Hopf algebras in positive characteristic, \emph{Proc.~Amer.~Math.~Soc.}, 137 (2009), 1925--1932.

\bibitem{Mon} S.~Montgomery, \emph{Hopf Algebras and Their Actions on Rings}, CBMS Regional Conference Series in Mathematics 82, American Mathematical Society, Providence, Rhode Island, 1993.

\bibitem{NW} V.~Nguyen and X.-T.~Wang, Pointed $p^3$-dimensional Hopf algebras in positive characteristic, preprint, arXiv:1609.03952.

\bibitem{NWW} V.~Nguyen, L.-H.~Wang and X.-T.~Wang, Primitive deformations of quantum $p$-groups, preprint, arXiv:1505.02454.

\bibitem{NZ} W.~D.~Nichols and M.~B.~Zoeller, A Hopf algebra freeness theorem, \emph{Amer.~J.~of~Math.}, 111 (1989), 381--385.


\bibitem{Shi} K.~Shimizu, On indicators of Hopf algebras, \emph{Israel~J.~Math.}, 207 (2015), 155--201.


\bibitem{Wang} X.-T.~Wang, Connected Hopf algebras of dimension $p^2$, \emph{J.~Algebra}, 391 (2013), 93--113.

\bibitem{WW} L.-H.~Wang and X.-T.~Wang, Classification of pointed Hopf algebras of dimension $p^2$ over any algebraically closed field, \emph{Algebr.~Represent.~Theory}, 17 (2014), 1267--1276. 




\bibitem{Wang2014} X.-T.~Wang, Local criteria for cocommutative Hopf algebras, \emph{Comm.~Algebra} 42 (2014), 5180--5191.



\end{thebibliography}
\end{document}